\documentclass[12pt]{amsart}

\usepackage{ucs}

\usepackage{amssymb}
\usepackage{amsthm}
\usepackage{amsmath}
\usepackage{latexsym}
\usepackage[cp1251]{inputenc}
\usepackage{graphicx}
\usepackage{wrapfig}
\usepackage{caption}
\usepackage{subcaption}
\usepackage{indentfirst}
\usepackage[left=2.5cm,right=2.4cm,top=2.4cm,bottom=2.5cm,bindingoffset=0cm]{geometry}
\usepackage{enumerate}
\usepackage{makecell}

\DeclareMathOperator{\aut}{Aut}

\DeclareMathOperator{\cay}{Cay}

\DeclareMathOperator{\PSL}{PSL}

\DeclareMathOperator{\rk}{rk}

\DeclareMathOperator{\Span}{Span}

\DeclareMathOperator{\sym}{Sym}
\DeclareMathOperator{\rad}{rad}

\DeclareMathOperator{\dimwl}{dim_{WL}}
\DeclareMathOperator{\rkwl}{rk_{WL}}
\DeclareMathOperator{\WL}{WL}

\def\r{\mathrm{r}}

\makeatletter 
\def\@seccntformat#1{\csname the#1\endcsname. } 
\def\@biblabel#1{#1.}

\makeatother

\title{On WL-rank and WL-dimension of some Deza dihedrants}

\author{Grigory Ryabov}
\address{Sobolev Institute of Mathematics, Novosibirsk, Russia}
\address{Novosibirsk State Technical University, Novosibirsk, Russia}
\address{St. Petersburg Department of V.A. Steklov Institute of Mathematics}
\address{Leonard Euler International Mathematical Institute in Saint Petersburg}
\email{gric2ryabov@gmail.com}

\author{Leonid Shalaginov}
\address{Chelyabinsk State University, Chelyabinsk, Russia}
\email{44sh@mail.ru}

\thanks{The first author is supported by Leonard Euler International Mathematical Institute in Saint Petersburg under agreement No.~075-15-2019-1620 with the Ministry of Science and Higher Education of the Russian Federation}

\date{}

\newtheorem{state}{Statement}[section]

\newtheorem{lemm}[state]{Lemma}
\newtheorem{theo}[state]{Theorem}

\newtheorem{rem}[state]{Remark}

\newtheorem{corl}{Corollary}[section]

\theoremstyle{definition}

\begin{document}

\vspace{\baselineskip}
\vspace{\baselineskip}

\vspace{\baselineskip}

\vspace{\baselineskip}

\begin{abstract}
The \emph{WL-rank} of a graph $\Gamma$ is defined to be the rank of the coherent configuration of $\Gamma$. The \emph{WL-dimension} of $\Gamma$ is defined to be the smallest positive integer~$m$ for which $\Gamma$ is identified by the $m$-dimensional Weisfeiler-Leman algorithm. We establish that some families of strictly Deza dihedrants have WL-rank~$4$ or~$5$ and WL-dimension~$2$. Computer calculations imply that every strictly Deza dihedrant with at most~$59$ vertices is circulant or belongs to one of the above families. We also construct a new infinite family of strictly Deza dihedrants whose WL-rank is a linear function of the number of vertices.
\\
\\
\textbf{Keywords}: WL-rank, WL-dimension, Deza graphs, Cayley graphs, dihedral group.
\\
\\
\textbf{MSC}: 05C25, 05C60, 05C75. 
\end{abstract}

\maketitle

\section{Introduction}

A \emph{coherent configuration} $\mathcal{X}$ on a finite set $V$ can be thought as a special partition of $V\times V$ such that the diagonal of $V\times V$ is a union of some classes (see~\cite[Definition~2.1.3]{CP}). The Weisfeiler-Leman algorithm~\cite{WeisL} computes efficiently for a given graph\footnote{All graphs in the paper are assumed to be undirected and without loops and multiple edges.} $\Gamma$ with vertex set $V$ and edge set $E$ the smallest coherent configuration $\WL(\Gamma)$ on $V$ such that $E$ is a union of some classes of $\WL(\Gamma)$. The coherent configuration $\WL(\Gamma)$ is called the \emph{WL-closure} of $\Gamma$. The studying of $\WL(\Gamma)$ can help to obtain the results about the graph $\Gamma$. For example, a several number of results on the isomorphism problem for some classes of graphs~\cite{GNP,P,PR} were obtained by studying of the corresponding coherent configurations.

The number of classes in the coherent configuration $\mathcal{X}$ is called the \emph{rank} of $\mathcal{X}$. The \emph{WL-rank} of the graph $\Gamma$ is defined to be the rank of $\WL(\Gamma)$. Clearly, $\rkwl(\Gamma)\leq |V|^2$. Observe that $\rkwl(\Gamma)\geq 2$ unless $|V|=1$ because the diagonal of $V\times V$ is a union of some classes of any coherent configuration on $V$. If $\Gamma$ is vertex-transitive, then $\rkwl(\Gamma)\leq |V|$.

A $k$-regular graph $\Gamma$ is called  \emph{strongly regular} if there exist nonnegative integers $\lambda$ and $\mu$ such that every two adjacent vertices have $\lambda$ common neighbors and every two nonadjacent vertices have $\mu$ common neighbors. The following generalization of strongly regular graphs going back to~\cite{Deza} was suggested in~\cite{EFHHH}. A $k$-regular graph $\Gamma$ on $n$ vertices is called a \emph{Deza} graph if there exist nonnegative integers $a$ and $b$ such that any pair of distinct vertices of $\Gamma$ has either $a$ or $b$ common neighbors. The numbers $(n,k,b,a)$ are called the \emph{parameters} of $\Gamma$. Clearly, if $a>0$ and $b>0$, then $\Gamma$ has diameter~$2$. A Deza graph is called a \emph{strictly} Deza graph if it is not strongly regular and has diameter~$2$. Deza graphs have been actively studied during the last years. For more details on Deza graphs and the recent progress in their studying, we refer the readers to the survey paper~\cite{GSh2}.

The WL-rank of a strongly regular graph is at most~$3$. So it seems natural to ask how large the WL-rank of a (strictly) Deza graph can be. This question for strictly Deza circulants was studied in~\cite{BPR}. From the results of this paper it follows that the WL-rank of every known strictly Deza circulant is at most~$6$. A circulant can be thought as a Cayley graph over a cyclic group. Recall that if $G$ is a finite group and $S$ is an identity-free inverse-closed subset of $G$, then the \emph{Cayley graph} $\cay(G,S)$ is defined to be the graph with vertex set $G$ and edge set $\{(g,sg):~s\in S,~g\in G\}$. The automorphism group of every Cayley graph over $G$ contains the group $G_{\r}$ of all right translations of $G$. So every Cayley graph is vertex-transitive and hence $\rkwl(\cay(G,S))\leq |G|$. In~\cite{CR}, it was constructed an infinite family of Cayley graphs $\cay(G,S)$ such that $\rkwl(\cay(G,S))=|G|$.

The class of Deza graphs looks very wide and the problem of the classification of all even strictly Deza graphs seems to be hopeless. One of the possible steps towards the classification of all strictly Deza graphs is the classification of strictly Deza graphs of small WL-rank. All strictly Deza circulants of WL-rank~$4$ were classified in~\cite{BPR}. Some constructions of strictly Deza graphs of WL-rank~$4$ were suggested in~\cite{GPSh2}.

By a \emph{dihedrant}, we mean a graph isomorphic to a Cayley graph over a dihedral group. All strongly regular dihedrants were described in~\cite{MiP}. In the present paper, we study strictly Deza dihedrants. In the first statement, we describe some families of strictly Deza dihedrants of WL-rank~$4$. The complete graph with $n$ vertices is denoted by $K_n$. The Cartesian product of graphs $\Gamma_1$ and $\Gamma_2$ is denoted by $\Gamma_1\times \Gamma_2$. If $A$ is a cyclic group of order~$n$, $D\subset A$ is a difference set in $A$ (see Section~2.4 for the definition), $G$ is a dihedral group of order~$2n$ containing $A$, and $x\in G\setminus A$ is an element of order~$2$, then put $\Gamma(D)=\cay(G, A^\#\cup xD)$, where $A^\#$ is the set of all nontrivial elements of $A$. Note that $\Gamma(D)$ is a complement to the graph from~\cite[Theorem~1.3(i)]{MiP}.

\begin{theo}\label{main1}
Each of the following graphs is a strictly Deza dihedrant of WL-rank~$4:$
\\
\\
$(1)$ $\Gamma(D)$, where $D$ is a cyclic difference set with parameters $(n,\frac{2n-1-\sqrt{8n-7}}{2},n+1-\sqrt{8n-7})$;
\\
\\
$(2)$ $K_4 \times K_{m}$, where $m\geq 2$ is not divisible by~$4$.
\end{theo}

Observe that $K_4\times K_m$ is a strictly Deza graph for every~$m$ by~\cite[Theorem~2.8(ii)]{EFHHH}. However, if $m$ is divisible by~$4$, then $K_4\times K_m$ is not a dihedrant (see Lemma~\ref{pr13}). From~\cite[Theorem~1.1]{BPR} it follows that $\Gamma(D)$ is not a circulant and $K_4 \times K_{m}$ is circulant if and only if $m$ is odd. The computer calculations~\cite{GPSh,GSh} show that Theorem~\ref{main1} describes all strictly Deza dihedrants with small number of vertices which are not circulants.

\begin{corl}\label{corl1}
A graph with at most~$59$ vertices is a strictly Deza dihedrant of WL-rank~$4$ if and only if it is isomorphic to a strictly Deza circulant of WL-rank~$4$ with even number of vertices or to one of the graphs from Theorem~\ref{main1}.
\end{corl}

We do not know whether there exist infinitely many cyclic difference sets with parameters $(n,\frac{2n-1-\sqrt{8n-7}}{2},n+1-\sqrt{8n-7})$. If $n\leq 100$ then there exist three such difference sets with parameters $(7,3,1)$, $(11,6,3)$, and $(37,28,21)$ (see~\cite{Bau}). However, using graphs $\Gamma(D)$, one can construct infinite family of strictly Deza dihedrants of WL-rank~$5$ which is described in the next statement. The lexicographic products of graphs $\Gamma_1$ and $\Gamma_2$ is denoted by $\Gamma_1[\Gamma_2]$.

\begin{theo}\label{main2}
Each of the following graphs is a strictly Deza dihedrant of WL-rank~$5:$
\\
\\
$(1)$ $K_m[\Gamma(D)]$, where $D$ is a cyclic difference set with parameters $(n,\frac{2n-1-\sqrt{8n-7}}{2},n+1-\sqrt{8n-7})$ and $m\geq 2$;
\\
\\
$(2)$ $K_m[K_4\times K_2]$, where $m\geq 2$.
\end{theo}

Note that the graph $K_m[K_4\times K_l]$ is not a strictly Deza graph whenever $l\neq 2$ and $m\geq 2$ (see Lemma~\ref{pr23}). The computer calculations~\cite{GPSh,GSh} yield that Theorem~\ref{main2} describes all strictly Deza dihedrants of WL-rank~$5$ with small number of vertices which are not circulants. Observe that by the results from~\cite{BPR}, there is a unique (up to isomorphism) strictly Deza circulant of WL-rank~$5$ with even number of vertices not exceeding~$59$; it has~$8$ vertices.

\begin{corl}\label{corl2}
A graph with at most~$59$ vertices is a strictly Deza dihedrant of WL-rank~$5$ if and only if it is isomorphic to a strictly Deza circulant of WL-rank~$5$ with~$8$ vertices or to one of the graphs from Theorem~\ref{main2}.
\end{corl}

Recall that every known strictly Deza circulant has WL-rank at most~$6$. It seems natural to ask the following question: can the WL-rank of a strictly Deza dihedrant be arbitrary large? The next statement give a positive answer to this question.

\begin{theo}\label{main3}
Let $k\geq 3$ be an odd integer. There exists a strictly Deza dihedrant $\Sigma(k)$ with $8k$ vertices of WL-rank~$6k$.
\end{theo}

In this paper, we study one more parameter of graphs concerned with the Weisfeiler-Leman algorithm. The \emph{WL-dimension} (the \emph{Weisfeiler-Leman dimension}) $\dimwl(\Gamma)$ of a graph $\Gamma$ is defined to be the smallest positive integer~$d$ for which $\Gamma$ is identified by the $d$-dimensional Weisfeiler-Leman algorithm~\cite[p.~6]{GN} (see also~\cite{Grohe}). If $\dimwl(\Gamma)\leq d$, then one can verify whether $\Gamma$ and any other graph are isomorphic in time $n^{O(d)}$ using the $d$-dimensional Weisfeiler-Leman algorithm. More details on the WL-dimension of graphs can be found in~\cite{Grohe,GN}. In~\cite{BPR}, it was proved that the WL-dimension of every known Deza circulant is at most~$3$. The next statement is concerned with the WL-dimension of strictly Deza dihedrants which appear in Theorems~\ref{main1}, \ref{main2}, and~\ref{main3}. If $D$ is a difference set in a group $G$, then the order of $G$ is denoted by $v(D)$ (see Section~2.4).

\begin{theo}\label{main4}
Each of the following graphs has WL-dimension~$2:$
\\
\\
$(1)$ $K_m[\Gamma(D)]$, where $m\geq 1$ and $v(D)\leq 13$;
\\
\\
$(2)$ $K_m[K_4\times K_2]$, where $m\geq 1$;
\\
\\
$(3)$ $\Sigma(k)$, where $k\geq 3$ is an odd integer.
\end{theo}

From~\cite{Bau} it follows that there is a cyclic difference set $D$ with $v(D)\in\{7,11,13\}$ and there is no a cyclic difference set $D$ with $v(D)\leq 13$ and $v(D)\notin\{7,11,13\}$. If $v(D)>13$, then $\dimwl(K_m[\Gamma(D)])>2$ (see Remark~\ref{wldim}).

The text of the paper is organized in the following way. The WL-closure of a Cayley graph can be thought as an $S$-ring. Section~$2$ contains all necessary definitions and statements on $S$-rings, Cayley graphs,WL-rank, WL-dimension, and difference sets. In Sections~$3$, $4$, $5$, and~$6$, we prove Theorems~\ref{main1}, \ref{main2}, \ref{main3}, and~\ref{main4}, respectively. In Appendix, there is a table in which we collect properties of graphs from the paper.

The authors would like to thank D.~Churikov for the help with computer calculations.

\section{Preliminaries}

In this section we provide a background of $S$-rings, Cayley graphs, and difference sets. In general, we follow to~\cite{BPR,Ry2}, where the most of definitions and statements is contained. Throughout the paper, the symmetric group of a finite set~$\Omega$ is denoted by $\sym(\Omega)$; if $|\Omega|=n$ and the set $\Omega$ is not important, then we write $\sym(n)$ instead of $\sym(\Omega)$.

\subsection{$S$-rings}
Let $G$ be a finite group and $\mathbb{Z}G$  the integer group ring. The identity element of $G$ and the set of all nonidentity elements of $G$ are denoted by~$e$ and~$G^\#$, respectively. If $X\subseteq G$ then the element $\sum \limits_{x\in X} {x}$ of the group ring $\mathbb{Z}G$ is denoted by~$\underline{X}$. The explicit computation shows that $\underline{G}^2=|G|\underline{G}$. The set $\{x^{-1}:x\in X\}$ is denoted by $X^{-1}$.

A subring  $\mathcal{A}\subseteq \mathbb{Z} G$ is called an \emph{$S$-ring} (a \emph{Schur} ring) over $G$ if there exists a partition $\mathcal{S}=\mathcal{S}(\mathcal{A})$ of~$G$ such that:

$(1)$ $\{e\}\in\mathcal{S}$;

$(2)$  if $X\in\mathcal{S}$ then $X^{-1}\in\mathcal{S}$;

$(3)$ $\mathcal{A}=\Span_{\mathbb{Z}}\{\underline{X}:\ X\in\mathcal{S}\}$.

\noindent The notion of an $S$-ring goes back to Schur~\cite{Schur} and Wielandt~\cite{Wi}. The elements of $\mathcal{S}$ are called the \emph{basic sets} of  $\mathcal{A}$ and the number $\rk(\mathcal{A})=|\mathcal{S}|$ is called the \emph{rank} of~$\mathcal{A}$. The $S$-ring of rank~$2$ over $G$ defined by the partition $\{\{e\},G^\#\}$ is denoted by $\mathcal{T}_G$.

The following easy lemma can be found, e.g., in~\cite[Lemma~2.4]{Ry1}.

\begin{lemm}\label{basicset}
Let $\mathcal{A}$ be an $S$-ring over a group $G$. If $X,Y\in \mathcal{S}(\mathcal{A})$ then $XY\in \mathcal{S}(\mathcal{A})$ whenever $|X|=1$ or $|Y|=1$.
\end{lemm}

\begin{lemm}~\cite[Lemma~2.2]{CR}\label{groupring}
Let $\mathcal{A}$ be an $S$-ring over a group $G$ and $X\subseteq G$ such that $\langle X \rangle=G$. Suppose that $\{x\}\in \mathcal{S}(\mathcal{A})$ for every $x\in X$. Then $\mathcal{A}=\mathbb{Z}G$.
\end{lemm}

A set $X \subseteq G$ is called an \emph{$\mathcal{A}$-set} if $\underline{X}\in \mathcal{A}$ or, equivalently, $X$ is a union of some basic sets of $\mathcal{A}$. The set of all $\mathcal{A}$-sets is denoted by $\mathcal{S}^*(\mathcal{A})$. Clearly, if $X\in \mathcal{S}^*(\mathcal{A})$ and $|X|=1$, then $X\in \mathcal{S}(\mathcal{A})$. Let $\mathcal{A}_1$ and $\mathcal{A}_2$ be $S$-rings over $G$. It is easy to see that $\mathcal{A}_1\leq \mathcal{A}_2$, i.e. $\mathcal{A}_1$ is a subring of $\mathcal{A}_2$, if and only if every basic set of $\mathcal{A}_1$ is an $\mathcal{A}_2$-set.

One can verify (see~\cite[Eq.~(1)]{CR}) that if $X,Y\in \mathcal{S}^*(\mathcal{A})$, then 
$$X\cap Y,X\cup Y, X\setminus Y, Y\setminus X, XY\in \mathcal{S}^*(\mathcal{A}).~\eqno(1)$$ 
A subgroup $H \leq G$ is called an \emph{$\mathcal{A}$-subgroup} if $H\in \mathcal{S}^*(\mathcal{A})$. For every $X\in \mathcal{S}^*(\mathcal{A})$, the groups $\langle X \rangle$ and $\rad(X)=\{g\in G:~Xg=gX=X\}$ are $\mathcal{A}$-subgroups.

Let $L \unlhd U\leq G$. A section $U/L$ is called an \emph{$\mathcal{A}$-section} if $U$ and $L$ are $\mathcal{A}$-subgroups. If $U/L$ is an $\mathcal{A}$-section then the module
$$\mathcal{A}_{U/L}=Span_{\mathbb{Z}}\left\{\underline{X}^{\pi}:~X\in\mathcal{S}(\mathcal{A}),~X\subseteq U\right\},$$
where $\pi:U\rightarrow U/L$ is the canonical epimorphism, is an $S$-ring over $U/L$.

The following lemma is known as the Wielandt's principle~\cite[Proposition~22.1]{Wi}.

\begin{lemm}\label{sw}
Let $\mathcal{A}$ be an $S$-ring over $G$, $\xi=\sum \limits_{g\in G} c_g g\in \mathcal{A}$, where $c_g\in \mathbb{Z}$, and $c\in \mathbb{Z}$. Then $\{g\in G:~c_g=c\}\in \mathcal{S}^*(\mathcal{A})$.
\end{lemm}

Let $L$ be a normal subgroup of $G$, $\mathcal{A}_1$ an $S$-ring over $L$, and $\mathcal{A}_{2}$ an $S$-ring over $G/L$. Then the partition
$$\{X:~X\in \mathcal{S}(\mathcal{A}_1)\}\cup \{X^{\pi^{-1}}:~X\in \mathcal{S}(\mathcal{A}_{2})\},$$
where $\pi:G\rightarrow G/L$ is the canonical epimorphism, defines the $S$-ring $\mathcal{A}$ over $G$ which is called the \emph{wreath product} of $\mathcal{A}_1$ and $\mathcal{A}_{2}$ and denoted by $\mathcal{A}_1\wr \mathcal{A}_{2}$. Clearly, $\mathcal{A}_L=\mathcal{A}_1$ and $\mathcal{A}_{G/L}=\mathcal{A}_2$.

Let $U/L$ be an $\mathcal{A}$-section of $G$. The $S$-ring~$\mathcal{A}$ is called the \emph{$U/L$-wreath product} or \emph{generalized wreath product} of $\mathcal{A}_U$ and $\mathcal{A}_{G/L}$ if $L$ is normal in $G$ and $L\leq\rad(X)$ for each basic set $X$ outside~$U$. In this case, we write $\mathcal{A}=\mathcal{A}_U\wr_{U/L}\mathcal{A}_{G/L}$. If $L>\{e\}$ and $U<G$ then the $U/L$-wreath product is called \emph{nontrivial}. The notion of the generalized wreath product of $S$-rings was introduced in~\cite{EP1}. The following equality can be found, e.g., in~\cite{CR}:
$$\rk(\mathcal{A}_U\wr_{U/L}\mathcal{A}_{G/L})=\rk(\mathcal{A}_U)+\rk(\mathcal{A}_{G/L})-\rk(\mathcal{A}_{U/L}).~\eqno(2)$$
If $U=L$ then the $U/L$-wreath product coincides with $\mathcal{A}_L\wr \mathcal{A}_{G/L}$.

Let $H$ and $V$ be $\mathcal{A}$-subgroups such that $G=H\times V$. The $S$-ring~$\mathcal{A}$ is called the \emph{tensor product} of $\mathcal{A}_H$ and $\mathcal{A}_V$ if 
$$\mathcal{S}(\mathcal{A})=\{X_1\times X_2:~X_1\in\mathcal{S}(\mathcal{A}_H),~X_2\in \mathcal{S}(\mathcal{A}_V)\}.$$
In this case we write $\mathcal{A}=\mathcal{A}_H\otimes \mathcal{A}_V$.

If $X\subseteq G$, then the edge set of the Cayley graph $\cay(G,X)$ is denoted by $E(X)$. Let $\mathcal{A}_1$ and $\mathcal{A}_2$ be $S$-rings over groups $G_1$ and $G_2$, respectively. A bijection $f:G_1\rightarrow G_2$ is called an \emph{isomorphism} from $\mathcal{A}_1$ to $\mathcal{A}_2$ if 
$$\{E(X_1)^f:~X_1\in \mathcal{S}(\mathcal{A}_1)\}=\{E(X_2):~X_2\in \mathcal{S}(\mathcal{A}_2)\},$$
where $E(X_1)^f=\{(x^f,y^f):~(x,y)\in E(X_1)\}$. If there exists an isomorphism from $\mathcal{A}_1$ to $\mathcal{A}_2$, then we say that $\mathcal{A}_1$ and $\mathcal{A}_2$ are \emph{isomorphic} and write $\mathcal{A}_1\cong \mathcal{A}_2$.

The \emph{automorphism group} $\aut(\mathcal{A})$ is defined to be the group 
$$\bigcap \limits_{X\in \mathcal{S}(\mathcal{A})} \aut(\cay(G,X)).$$
Since $\aut(\cay(G,X))\geq G_{\r}$ for every $X\in \mathcal{S}(\mathcal{A})$, we conclude that $\aut(\mathcal{A})\geq G_{\r}$. Clearly, if $\rk(\mathcal{A})=2$, then $\aut(\mathcal{A})=\sym(G)$. It is easy to check that $\aut(\mathcal{A})=G_{\r}$ if and only if $\mathcal{A}=\mathbb{Z}G$.

\begin{lemm}\label{aut}
Let $\mathcal{A}_1$ and $\mathcal{A}_2$ be $S$-rings and $\mathcal{A}=\mathcal{A}_1\ast \mathcal{A}_2$, where $\ast\in\{\otimes,\wr\}$. Then $\aut(\mathcal{A})=\aut(\mathcal{A}_1)\ast \aut(\mathcal{A}_2)$.
\end{lemm}

\begin{proof}
The statement of the lemma follows from~\cite[Eq.~3.2.18]{CP} whenever $\ast=\otimes$ and from~\cite[Eq.~3.4.11]{CP} whenever $\ast=\wr$.
\end{proof}

A cyclic group of order~$n$ is denoted by $C_n$.

\begin{lemm}\label{plus1}
Let $H<G$ and $\mathcal{S}_0$ a partition of $H$ which defines an $S$-ring $\mathcal{A}_0$ over $H$. Then the partition $\mathcal{S}_0\cup \{G\setminus H\}$ defines the $S$-ring $\mathcal{A}$ over $G$ which is isomorphic to $\mathcal{A}_0\wr \mathcal{T}_{C_m}$, where $m=|G:H|$.
\end{lemm}

\begin{proof}
Put $Y=G\setminus H$. Since $H\leq\rad(Y)$, we conclude that $\underline{X}\underline{Y}=\underline{Y}\underline{X}=|X|\underline{Y}$ for every $X\in\mathcal{S}_0$. From~\cite[Eq.~(3)]{BPR} it follows that
$$\underline{Y}^2=(\underline{G}-\underline{H})^2=(|G|-2|H|)\underline{Y}+(|G|-|H|)\underline{H}.$$
Therefore the partition $\mathcal{S}$ defines the $S$-ring $\mathcal{A}$ over $G$ such that $\mathcal{A}_H=\mathcal{A}_0$. Since $\rk(\mathcal{A})=\rk(\mathcal{A}_0)+1$, we obtain $\mathcal{A}\cong \mathcal{A}_0\wr \mathcal{T}_{C_m}$, where $m=|G:H|$, by~\cite[Corollary~3.3]{MP}.
\end{proof}

\subsection{Cayley graphs}
Let $S\subseteq G$, $e\notin S$, $S=S^{-1}$, and $\Gamma=\cay(G,S)$. The \emph{WL-closure} $\WL(\Gamma)$ of $\Gamma$ can be thought as the smallest $S$-ring over $G$ such that $S\in\mathcal{S}^*(\mathcal{A})$ (see~\cite[Section~5]{BPR}). It is easy to see that the WL-rank of the complete graph is equal to~$2$. If $\mathcal{A}=\WL(\Gamma)$, then $\rkwl(\Gamma)=\rk(\mathcal{A})$ by~\cite[Lemma~5.1]{BPR}. From~\cite[Theorem~2.6.4]{CP} it follows that $\aut(\Gamma)=\aut(\mathcal{A})$.

\begin{lemm}\cite[Lemma~5.2]{BPR}\label{deza}
Let $G$ be a group of order~$n$, $S\subseteq G$ such that $e\notin S$, $S=S^{-1}$, and $|S|=k$, and $\Gamma=\cay(G,S)$. The graph $\Gamma$ is a Deza graph with parameters $(n,k,b,a)$ if and only if $\underline{S}^2=ke+a\underline{S_a}+b\underline{S_b}$, where $S_a\cup S_b=G^\#$ and $S_a\cap S_b=\varnothing$. Moreover, $\Gamma$ is strongly regular if and only if $S_a=S$ or $S_b=S$.
\end{lemm}

Let $\Gamma_1=(V_1,E_1)$ and $\Gamma_2=(V_2,E_2)$ be graphs. The \emph{Cartesian product} $\Gamma_1\times \Gamma_2$ of $\Gamma_1$ and $\Gamma_2$ is defined to be the graph with vertex set $V=V_1\times V_2$ and edge set $E$ defined as follows:
$$((v_1,v_2),(u_1,u_2))\in E~\text{if and only if}~v_1=u_1~\text{and}~(v_2,u_2)\in E_2~\text{or}~v_2=u_2~\text{and}~(v_1,u_1)\in E_1.$$
The \emph{lexicographic product} $\Gamma_1[\Gamma_2]$ of $\Gamma_1$ and $\Gamma_2$ is defined to be the graph with vertex set $V=V_1\times V_2$ and edge set $E$ defined as follows:
$$((v_1,v_2),(u_1,u_2))\in E~\text{if and only if}~(v_1,u_1)\in E_1~\text{or}~v_1=u_1~\text{and}~(v_2,u_2)\in E_2.$$
Note that if $\Gamma_1$ has one vertex, then $\Gamma_1[\Gamma_2]\cong \Gamma_2$.

\begin{lemm}\cite[Lemma~5.4]{BPR}\label{lexproduct}
Let $G$ be a group, $H$ a normal subgroup of~$G$, $\pi:G\rightarrow G/H$ the canonical epimorphism, and $\overline{G}=G^{\pi}$. Suppose that $\Gamma_1=\cay(\overline{G},\overline{T})$ and $\Gamma_2=\cay(H,S)$ are Cayley graphs over $\overline{G}$ and $H$, respectively. Then 
$$\Gamma_1[\Gamma_2]\cong \cay(G,(\overline{T})^{\pi^{-1}}\cup S).$$
\end{lemm}

\begin{lemm}\label{wrclosure}
In the notations of Lemma~\ref{lexproduct}, let $\Gamma=\cay(G,(\overline{T})^{\pi^{-1}}\cup S)$. Then $\WL(\Gamma)\leq \WL(\Gamma_2)\wr \WL(\Gamma_1)$.
\end{lemm}

\begin{proof}
Put $\mathcal{A}_1=\WL(\Gamma_1)$, $\mathcal{A}_2=\WL(\Gamma_2)$, and $\mathcal{A}=\mathcal{A}_2\wr \mathcal{A}_1$. By the definition of the WL-closure, $\overline{T}\in \mathcal{S}^*(\mathcal{A}_1)$. So $(\overline{T})^{\pi^{-1}}$ is a union of some basic sets of $\mathcal{A}$ outside $H$. Since $S\in \mathcal{S}^*(\mathcal{A}_2)$, we conclude that $S\in \mathcal{S}^*(\mathcal{A})$ by the definition of the wreath product. Therefore $(\overline{T})^{\pi^{-1}}\cup S \in\mathcal{S}^*(\mathcal{A})$ and hence $\WL(\Gamma)\leq \mathcal{A}$.
\end{proof}

\begin{rem}
The inequality $\WL(\Gamma)\leq \WL(\Gamma_2)\wr \WL(\Gamma_1)$ in Lemma~\ref{wrclosure} can be strict. If $\Gamma_1$ and $\Gamma_2$ are complete graphs then $\Gamma=\Gamma_1[\Gamma_2]$ is complete and hence $\rk(\WL(\Gamma))=2$. On the other hand, $\rk(\WL(\Gamma_1))=\rk(\WL(\Gamma_2))=2$. Eq.~(2) implies that $\rk(\WL(\Gamma_2)\wr \WL(\Gamma_1))=3$. Therefore $\WL(\Gamma)< \WL(\Gamma_2)\wr \WL(\Gamma_1)$.
\end{rem}

\begin{lemm}\label{cliqueext}
Let $\Gamma$ be a Deza graph with parameters $(n,k,b,a)$, $a\neq b$, and $m$ a positive integer. Then $K_m[\Gamma]$ is a Deza graph if and only if $a=2k-n$ or $\beta=2k-n$. In this case, $K_m[\Gamma]$ has parameters $(mn,k+(m-1)n,b+(m-1)n,a+(m-1)n)$. 
\end{lemm}

\begin{proof}
Since $K_m$ is strongly regular, \cite[Proposition~2.3]{EFHHH} implies that $K_m[\Gamma]$ is $k+(m-1)n$-regular and the number of common neighbors of two vertices of $K_m[\Gamma]$ belongs to the set $\{a+(m-1)n,b+(m-1)n,(m-2)n+2k\}$. So $K_m[\Gamma]$ is a Deza graph if and only if $|\{a+(m-1)n,b+(m-1)n,(m-2)n+2k\}|=2$. In view of $a\neq b$, the latter holds if and only if $a=2k-n$ or $\beta=2k-n$. 
\end{proof}

\subsection{WL-dimension}

The \emph{WL-dimension} (the \emph{Weisfeiler-Leman dimension}) $\dimwl(\Gamma)$ of a graph $\Gamma$ is defined to be the smallest positive integer~$d$ for which $\Gamma$ is identified by the $d$-dimensional Weisfeiler-Leman algorithm. If $\dimwl(\Gamma)\leq d$, then the isomorphism between $\Gamma$ and any other graph can be verified in time $n^{O(d)}$ using the Weisfeiler-Leman algorithm~\cite{WeisL}. The WL-dimension of Deza circulant graphs was studied in~\cite{BPR}. For more details on WL-dimension of graphs, we refer the readers to~\cite{Grohe,GN}.

Following~\cite[Section~4.2]{BPR}, we say that an $S$-ring $\mathcal{A}$ is \emph{separable} if the Cayley scheme $\mathcal{X}$ corresponding to $\mathcal{A}$ is separable, i.e. every algebraic isomorphism of $\mathcal{X}$ is induced by a combinatorial isomorphism. The exact definitions and more information on separable $S$-rings and schemes can be found in~\cite[Section~2.3.4]{CP} and~\cite{Ry2}. The following statement is a special case of~\cite[Theorem~2.5]{FKV}.

\begin{lemm}\label{dim2}
Let $\Gamma$ be a Cayley graph and $\mathcal{A}=\WL(\Gamma)$. Then $\dimwl(\Gamma)\leq 2$ if and only if $\mathcal{A}$ is separable.
\end{lemm}

\begin{lemm}\label{trivsep}
An $S$-ring $\mathcal{A}$ over a group $G$ is separable if one of the following statements holds:
\\
\\
$(1)$ $\rk(\mathcal{A})=2$;
\\
\\
$(2)$ $|G|\leq 14$;
\\
\\
$(3)$ $\mathcal{A}=\mathbb{Z}U\wr_{U/L} \mathbb{Z}(G/L)$ for some $\mathcal{A}$-section $U/L$ of $G$. 
\end{lemm}

\begin{proof}
If Statement~$1$ of the lemma holds, then $\mathcal{A}$ is separable by~\cite[Example~2.3.31]{CP}; if Statement~$2$ of the lemma holds, then $\mathcal{A}$ is separable by~\cite[p.~64]{CP}; if Statement~$3$ of the lemma holds, then $\mathcal{A}$ is separable by~\cite[Theorem~3.4.23]{CP}.
\end{proof}

\begin{lemm}\cite[Lemma~4.1]{BPR}\label{separ}
Let $\mathcal{A}_1$ and $\mathcal{A}_2$ be $S$-rings. The $S$-ring $\mathcal{A}_1*\mathcal{A}_2$, where $*\in\{\wr,\otimes\}$, is separable if and only if so are $\mathcal{A}_1$ and $\mathcal{A}_2$.
\end{lemm}

The following lemma immediately follows from the description of all regular graphs of WL-dimension~$1$~\cite[Lemma~3.1 (a)]{AKRV}.

\begin{lemm}\label{dim1}
If $\Gamma$ is a regular graph such that $\dimwl(\Gamma)=1$, then $\Gamma$ is strongly regular, i.e. $\rkwl(\Gamma)\leq 3$.
\end{lemm}

\subsection{Difference sets}

A subset $D$ of $G$ is called a \emph{difference set} in $G$ if 
$$\underline{D}^{-1}\underline{D}=ke+\lambda \underline{G}^\#,$$
where $k=k(D)=|D|$ and $\lambda=\lambda(D)$ is a positive integer. The numbers $(v,k,\lambda)$, where $v=v(D)=|G|$, are called the \emph{parameters} of $D$. It is easy to check that $G\setminus D$ is a difference set with parameters $(v,v-k,v-2k+\lambda)$. A simple counting argument implies that
$$\lambda=\frac{k^2-k}{v-1}.~\eqno(3)$$
If $G$ is a cyclic group, then $D$ is defined to be a \emph{cyclic} difference set. For the general theory of difference sets, we refer the readers to~\cite{Bau,Pott}.

\section{Proof of Theorem~\ref{main1}}

Let $n\geq 3$. A dihedral group of order~$2n$ is denoted by $D_{2n}$. Put $G=\langle x,y:~x^n=y^2=e,~x^y=x^{-1}\rangle\cong D_{2n}$, $A=\langle x \rangle$, and $B=\langle y \rangle$. Clearly, $G=A\rtimes B$. These notations are valid until the end of the paper. Suppose that $D$ is a difference set in $A$ with parameters $(n,k,\lambda)$. Then~\cite[Lemma~5.1]{PV} implies that the partition of $G$ into sets
$$\{e\}, A^\#, bD, b(A\setminus D)$$
defines the $S$-ring $\mathcal{A}=\mathcal{A}(D)$. Clearly, $\rk(\mathcal{A})=4$. Put 
$$S=A^\# \cup yD~\text{and}~\Gamma=\Gamma(D)=\cay(G,S).$$ 
One can see that $S=S^{-1}$ and $\Gamma$ is $(n-1+k)$-regular.

\begin{lemm}\label{pr11}
The graph $\Gamma$ is a strictly Deza graph if and only if $D$ has parameters 
$$(n,\frac{2n-1-\sqrt{8n-7}}{2},n+1-\sqrt{8n-7}).$$
\end{lemm}

\begin{proof}
 The straightforward computation in the group ring $\mathbb{Z}G$ using the equalities $yxy=x^{-1}$, $(\underline{A}^\#)^2=(n-1)e+(n-2)\underline{A}^\#$, and $\underline{D}^{-1}\underline{D}=ke+\lambda A^\#$, implies that
$$\underline{S}^2=(k+n-1)e+(n-2+\lambda)\underline{A}^\#+2(k-1)y\underline{D}+2ky\underline{A\setminus D}~\eqno(4)$$
From Lemma~\ref{deza} and Eq.~(4) it follows that $\Gamma$ is a Deza graph if and only if $|\{n-2+\lambda,2(k-1),2k\}|=2$ and $\Gamma$ is strongly regular if and only if $n-2+\lambda=2(k-1)$. Obviously, $2(k-1)\neq 2k$ and every element from $\{n-2+\lambda,2(k-1),2k\}$ is non-zero. Therefore $\Gamma$ is a strictly Deza graph if and only if $n-2+\lambda=2k$. Due to Eq.~(3), we have $n-2+\frac{k^2-k}{n-1}=2k$ and hence $k^2-k(2n-1)+n^2-3n+2=0$. Since $k\leq n$ and $n\geq 3$, we obtain
$$k=\frac{2n-1-\sqrt{8n-7}}{2}~\text{and}~\lambda=2k-n+2=n+1-\sqrt{8n-7}.$$
Thus, $\Gamma$ is a strictly Deza graph if and only if $D$ has parameters $(n,\frac{2n-1-\sqrt{8n-7}}{2},n+1-\sqrt{8n-7})$. 
\end{proof}

The next lemma immediately follows from the definition of $\Gamma(D)$, Lemma~\ref{deza}, Lemma~\ref{pr11}, and Eq.~(4).

\begin{lemm}\label{parameters}
If $\Gamma$ is a strictly Deza graph then $\Gamma$ has parameters $(2n,n-1+k,2k,2(k-1))$, where $k=\frac{2n-1-\sqrt{8n-7}}{2}$. 
\end{lemm}

\begin{lemm}\label{pr12}
In the above notations, $\WL(\Gamma)=\mathcal{A}$. In particular, $\rkwl(\Gamma)=\rk(\mathcal{A})=4$.
\end{lemm}

\begin{proof}
Put $\mathcal{A}^{\prime}=\WL(\Gamma)$. Let us prove that $\mathcal{A}^{\prime}=\mathcal{A}$. Since $S\in \mathcal{S}^*(\mathcal{A})$ and $\mathcal{A}^{\prime}$ is the smallest $S$-ring over $G$ such that $S\in \mathcal{S}^*(\mathcal{A}^{\prime})$, we conclude that $\mathcal{A}^{\prime}\leq \mathcal{A}$. From Lemma~\ref{sw} and Eq.~(4) it follows that $yD\in \mathcal{S}^*(\mathcal{A}^{\prime})$. Since $S,yD\in \mathcal{S}^*(\mathcal{A}^{\prime})$, Eq.~(1) implies that $A^\#=S\setminus yD\in \mathcal{S}^*(\mathcal{A}^{\prime})$ and $G^\#\setminus S\in \mathcal{S}^*(\mathcal{A}^{\prime})$. Therefore every basic set of $\mathcal{A}$ is an $\mathcal{A}^{\prime}$-set and hence $\mathcal{A}^{\prime}\geq \mathcal{A}$. Thus, $\mathcal{A}^{\prime}=\mathcal{A}$.
\end{proof}

From~\cite[Theorem~2.8]{EFHHH} it follows that $K_4\times K_m$ has parameters $(4m,m+2,m-2,2)$. The WL-closure of $K_4\times K_m$ is isomorphic to $\mathcal{T}_{C_4}\otimes \mathcal{T}_{C_m}$ by~\cite[Example~3.2.12]{CP}.

\begin{lemm}\label{pr13}
The graph $K_4 \times K_m$ is a dihedrant if and only if $m$ is not divisible by~$4$.
\end{lemm}

\begin{proof}
Observe that 
$$\aut(K_4\times K_m)=\aut(\WL(K_4\times K_m))\cong \aut(\mathcal{T}_{C_4}\otimes \mathcal{T}_{C_m})\cong \sym(4)\times \sym(m),$$ where the second equality holds by the remark before the lemma and the third equality holds by Lemma~\ref{aut}. By the Sabidussi's theorem, a graph is isomorphic to a Cayley graph over a given group if and only if the automorphism group of the graph contains a regular subgroup isomorphic to this group. Therefore $K_4\times K_m$ is isomorphic to a Cayley graph over $G\cong D_{2n}$, where $n=2m$, if and only if $\sym(4)\times \sym(m)$ contains a regular subgroup isomorphic to~$G$. The latter holds if and only if there exist subgroups $L$ and $U$ of $G$ such that $|L|=4$, $|U|=m$, $|L\cap U|=1$, and $G=LU$. 

Let $A_1$ be the Sylow $2$-subgroup of $A$ and $A_2$ the Hall $2^{\prime}$-subgroup of $A$. Clearly, $A=A_1\times A_2$. If $m$ is odd, then $|A_1|=2$ and $|A_2|=m$. In this case, one can take $L=A_1\times B$ and $U=A_2$. If $m=2l$, where $l$ is odd, then $|A_1|=4$ and $|A_2|=l$. In this case, one can take $L=A_1$ and $U=A_2\rtimes B$. Thus, if $m$ is not divisible by~$4$, then $K_4 \times K_m$ is isomorphic to a dihedrant.

Suppose that $m$ is divisible by~$4$ and $G=LU$, where $|L|=4$, $|U|=m$, and $|L\cap U|=1$. Clearly, $|U|$ is divisible by~$4$ in this case. So $U$ must contain a subgroup of $A$ order~$2$. On the other hand, $L$ also must contain a subgroup of $A$ order~$2$ because $|L|=4$. We obtain a contradiction to $|L\cap U|=1$. Thus, if $m$ is divisible by~$4$ then $K_4 \times K_m$ is not isomorphic to a dihedrant.
\end{proof}

\begin{proof}[Proof of Theorem~\ref{main1}]
Clearly, the graph $\Gamma(D)$ is a dihedrant. The WL-rank of $\Gamma(D)$ is equal to~$4$ by Lemma~\ref{pr12} and $\Gamma(D)$ is a strictly Deza graph whenever $D$ has parameters $(n,\frac{2n-1-\sqrt{8n-7}}{2},n+1-\sqrt{8n-7})$ by Lemma~\ref{pr11}.

The graph $K_4\times K_m$ is a strictly Deza graph by~\cite[Theorem~2.8]{EFHHH} for every $m\geq 2$. The WL-rank of $K_4\times K_m$ is equal to~$4$ by~\cite[Example~3.2.12]{CP} and $K_4\times K_m$ is isomorphic to a dihedrant whenever $m$ is not divisible by~$4$ by Lemmma~\ref{pr13}.
\end{proof}

\begin{proof}[Proof of Corollary~\ref{corl1}]
Every circulant graph with even number of vertices is isomorphic to a dihedrant by~\cite[Proposition~2.1]{MS}. So the ``if'' part of Corollary~\ref{corl1} follows from the above result and Theorem~\ref{main1}. The ``only if'' part of Corollary~\ref{corl1} follows from the computational results~\cite{GPSh,GSh}.
\end{proof}

\section{Proof of Theorem~\ref{main2}}

As in the previous section, $G=\langle x,y:~x^n=y^2=e,~x^y=x^{-1}\rangle\cong D_{2n}$. Let $A_0\leq A$, $G_0=A_0\rtimes B$, and $|G:G_0|=|A:A_0|=m\geq 2$. Suppose that $|A_0|=l$ and $D$ is a difference set of size $k$ in $A_0$ with parameters $(l,k,2k-l+2)$, where $k=\frac{2l-1-\sqrt{8l-7}}{2}$.  Put 
$$T=A_0^\#\cup yD \cup (G\setminus G_0)~\text{and}~\Delta=\Delta(D,m)=\cay(G,T).$$ 
One can see that $T=T^{-1}$ and $\Delta$ is $(2n-l+k-1)$-regular. The graph $\cay(G_0,A_0^\#\cup yD)$ is isomorphic to $\Gamma(D)$. So $\Delta$ is isomorphic to $K_m[\Gamma(D)]$ by Lemma~\ref{lexproduct}. 

\begin{lemm}\label{pr21}
The graph $\Delta$ is a strictly Deza graph with parameters $(2n,2n-l+k-1,n-l+2k,n-l+2k-2)$.
\end{lemm}

\begin{proof}
The parameters of $\Gamma(D)$ are equal to $(2l,l-1+k,2k,2(k-1))$ by Lemma~\ref{parameters}. Observe that $2(k-1)=2(l-1+k)-2l$.  Therefore $\Delta$ is a Deza graph with parameters $(2n,2n-l+k-1,n-l+2k,n-l+2k-2)$ by Lemma~\ref{cliqueext}. All parameters of $\Delta$ are positive because $n>l$. So $\Delta$ has diameter~$2$. Since $\Gamma(D)$ is not strongly regular, $\Delta$ is also not strongly regular. Therefore $\Delta$ is a strictly Deza graph.
\end{proof}

Let us consider the partition of $G$ into the following sets:
$$X_0=\{e\}, X_1=A_0^\#, X_2=yD, X_3=y(A_0\setminus D), X_4=G\setminus G_0.$$
The partition $\{X_0,X_1,X_2,X_3\}$ of $G_0$ defines the $S$-ring $\mathcal{A}(D)$ over $G_0$ by~\cite[Lemma~5.1]{PV}. So the partition $\{X_0,X_1,X_2,X_3,X_4\}$ defines the $S$-ring $\mathcal{B}=\mathcal{B}(D,m)$ over $G$ such that $\mathcal{B}\cong \mathcal{A}(D)\wr \mathcal{T}_{C_m}$ by Lemma~\ref{plus1}.

\begin{lemm}\label{pr22}
In the above notations, $\WL(\Delta)=\mathcal{B}$. In particular, $\rkwl(\Delta)=\rk(\mathcal{B})=5$.
\end{lemm}

\begin{proof}
Put $\mathcal{B}^{\prime}=\WL(\Delta)$. Note that $\mathcal{B}^{\prime}\leq \mathcal{B}$ by Lemma~\ref{wrclosure}. Since $T\in \mathcal{S}^*(\mathcal{B}^{\prime})$, we obtain $G^\#\setminus T=X_3\in \mathcal{S}^*(\mathcal{B}^{\prime})$ by Eq.~(1). The set $A_0\setminus D$ is a difference set in $A_0$. So $X_3X_3=A_0$ and hence $A_0^\#=X_1\in \mathcal{S}^*(\mathcal{B}^{\prime})$ by Eq.~(1). The group $G_0=\langle X_1 \cup X_3\rangle$ is a $\mathcal{B}^{\prime}$-subgroup. So $G\setminus G_0=X_4\in \mathcal{S}^*(\mathcal{B}^{\prime})$ by Eq.~(1). Finally, $X_2=T\setminus (X_1\cup X_4)\in \mathcal{S}^*(\mathcal{B}^{\prime})$ by Eq.~(1). Therefore every basic set of $\mathcal{B}$ is a $\mathcal{B}^{\prime}$-set. Thus, $\mathcal{B}^{\prime}\geq \mathcal{B}$ and hence $\mathcal{B}^{\prime}=\mathcal{B}$.
\end{proof}

\begin{lemm}\label{pr23}
Let $m\geq 2$ and $l\geq 1$. The graph $K_m[K_4\times K_l]$ is a strictly Deza graph if and only if $l=2$. The parameters of $K_m[K_4\times K_2]$ are equal to~$(8m,8m-4,8m-6,8m-8)$.
\end{lemm}

\begin{proof}
The parameters of $K_4\times K_l$ are equal to~$(4l,l+2,l-2,2)$ by~\cite[Theorem~2.8]{EFHHH}. Lemma~\ref{cliqueext} implies that $K_m[K_4\times K_l]$ is a Deza graph if and only if $4-2l=l-2$ or $4-2l=2$. The latter holds if and only if $l=2$ or $l=1$. If $l=1$, then $K_m[K_4\times K_l]$ is complete and hence not strictly Deza. The graph $K_m[K_4\times K_2]$ is a strictly Deza graph because $K_4\times K_2$ is not strongly regular. The parameters of $K_m[K_4\times K_2]$ are equal to~$(8m,8m-4,8m-6,8m-8)$ by Lemma~\ref{cliqueext}.
\end{proof}

Now suppose that $|A_0|=4$ and $x_0$ is a generator of $A_0$. Clearly, $n\geq 8$ and $G_0\cong D_8$ in this case. Put 
$$R=R_0\cup (G\setminus G_0)~\text{and}~\Lambda=\Lambda(m)=\cay(G,R),$$
where $R_0=\{x_0^2,yx_0,yx_0^2,yx_0^3\}$. One can see that $R=R^{-1}$ and $\Lambda$ is $(2n-4)$-regular. The straightforward check shows that the graph $\cay(G_0,R_0)$ is isomorphic to $K_4\times K_2$. Therefore $\Lambda$ is isomorphic to $K_m[K_4\times K_2]$ by Lemma~\ref{lexproduct}. Thus, $\Lambda$ is a strictly Deza graph with parameters~$(8m,8m-4,8m-6,8m-8)=(2n,2n-4,2n-6,2n-8)$.

Let us consider the partition of $G$ into the following sets:
$$Y_0=\{e\}, Y_1=\{yx_0^2\}, Y_2=\{y,x_0,x_0^3\}, Y_3=\{x_0^2,yx_0,yx_0^3\}, Y_4=G\setminus G_0.$$
The computation using~\cite{GAP} implies that the partition $\{Y_0,Y_1,Y_2,Y_3\}$ of $G_0$ defines an $S$-ring $\mathcal{C}_0$ over $G_0$ such that $\mathcal{C}_0=\WL(\cay(G_0,R_0))$. So by Lemma~\ref{plus1}, the partition $\{Y_0,Y_1,Y_2,Y_3,Y_4\}$ of $G$ defines the $S$-ring $\mathcal{C}=\mathcal{C}(m)$ over $G$ such that $\mathcal{C}\cong \mathcal{C}_0\wr \mathcal{T}_{C_m}$.

\begin{lemm}\label{pr24}
In the above notations, $\WL(\Lambda)=\mathcal{C}$. In particular, $\rkwl(\Lambda)=\rk(\mathcal{C})=5$.
\end{lemm}

\begin{proof}
Put $\mathcal{C}^{\prime}=\WL(\Lambda)$. By Lemma~\ref{wrclosure}, we have $\mathcal{C}^{\prime}\leq \mathcal{C}$. Eq.~(1) implies that $G\setminus R=Y_2\in \mathcal{S}^*(\mathcal{C}^{\prime})$. This yields that $G_0=\langle Y_2\rangle$ is a $\mathcal{C}^{\prime}$-subgroup and hence $G\setminus G_0=Y_4$ is a $\mathcal{C}^{\prime}$-set. The direct computation show that $Y_2Y_2=Y_3\cup \{e\}$ and hence $Y_3\in \mathcal{S}^*(\mathcal{C}^{\prime})$ by Eq.~(1). Finally, $Y_1=R\setminus (Y_2\cup Y_3)\in \mathcal{S}^*(\mathcal{C}^{\prime})$ by Eq.~(1). We proved that every basic set of $\mathcal{C}$ is a $\mathcal{C}^{\prime}$-set. Thus, $\mathcal{C}^{\prime}\geq \mathcal{C}$ and hence $\mathcal{C}^{\prime}=\mathcal{C}$.
\end{proof}

Theorem~\ref{main2} follows from Lemmas~\ref{pr21}, \ref{pr22}, \ref{pr23}, and~\ref{pr24}.

\begin{proof}[Proof of Corollary~\ref{corl2}]
The ``if'' part follows from~\cite[Proposition~2.1]{MS} and Theorem~\ref{main2}, whereas the ``only if'' part follows from the computational results~\cite{GPSh,GSh}.
\end{proof}

\section{Proof of Theorem~\ref{main3}}

As in the previous two sections, $G=\langle x,y:~x^n=y^2=e,~x^y=x^{-1}\rangle\cong D_{2n}$. Suppose that $n=4k$, where $k$ in an odd integer. The groups $\langle x^{2k} \rangle$ and $\langle x^4 \rangle$ are denoted by $A_0$ and $A_1$, respectively. Put $U=(A_0\times A_1)\rtimes B$. Clearly, $A_0\cong C_2$, $A_1\cong C_{k}$, and $U\cong D_{4k}$.  Put 
$$Z=x^2A_1\cup y(A_1\setminus\{x^{2(k-1)}\})\cup \{yx^{2k-1},yx^{-2},yx^{-1}\}~\text{and}~\Sigma=\Sigma(k)=\cay(G,Z).$$
One can see that $Z=Z^{-1}$ and $\Sigma$ is $2(k+1)$-regular.

\begin{lemm}\label{pr31}
The graph $\Sigma$ is a strictly Deza graph with parameters $(8k,2(k+1),2(k-1),2)$.
\end{lemm}

\begin{proof}
The straightforward computation in the group ring $\mathbb{Z}G$ using the equalities $\underline{A_1}^2=k\underline{A_1}$, $y\underline{A_1}=\underline{A_1}y$, and $yxy=x^{-1}$ implies that
$$\underline{Z}^2=2(k+1)e+2(k-1)(\underline{A_1}^\#+yx^2\underline{A_1})+2(x+x^2+x^3)\underline{A_1}+2y(e+x+x^3)\underline{A_1}.~\eqno(5)$$
From Lemma~\ref{deza} and Eq.~(5) it follows that $\Sigma$ is a strictly Deza graph with parameters $(8k,2(k+1),2(k-1),2)$.
\end{proof}

Let us consider the following partition of $G$:
$$\{\{g\}, gxA_0:~g\in U\}.$$
Observe that $g_1xA_0g_2xA_0\subseteq U$ whenever $g_1,g_2\in U$. So the above partition defines the $S$-ring $\mathcal{D}=\mathcal{D}(k)$ over $G$. One can see that $\mathcal{D}=\mathbb{Z}U \wr_{U/A_0} \mathbb{Z}(G/A_0)$.

\begin{lemm}\label{pr32}
In the above notations, $\WL(\Sigma)=\mathcal{D}$.
\end{lemm}

\begin{proof}
Put $\mathcal{D}^{\prime}=\WL(\Sigma)$. Let us prove that $\mathcal{D}^{\prime}=\mathcal{D}$. Observe that $Z\in \mathcal{S}^*(\mathcal{D})$ and hence $\mathcal{D}^{\prime}\leq \mathcal{D}$. Put $V=A_1^\#\cup yx^2A_1$. From Lemma~\ref{sw} and Eq.~(5) it follows that $V\in \mathcal{S}^*(\mathcal{D}^{\prime})$. So
$$\{yx^{-2}\}=Z\cap V\in\mathcal{S}(\mathcal{D}^{\prime})~\eqno(6)$$
by Eq.~(1). Eqs.~(1) and~(6) imply that
$$yx^{-2}Z\setminus Zyx^{-2}=\{x,x^{2k+1}\}=xA_0\in \mathcal{S}^*(\mathcal{D}^{\prime}).~\eqno(7)$$
Due to Eq.~(7), we conclude that $A_0=\rad(xA_0)$ and $A=\langle xA_0 \rangle$ are $\mathcal{D}^{\prime}$-subgroups. Note that
$$\mathcal{D}^{\prime}_{A_0}=\mathbb{Z}A_0~\eqno(8)$$
because $|A_0|=2$. Lemma~\ref{groupring} applied to $\mathcal{D}^{\prime}_{A/A_0}$ and Eq.~(7) yield that 
$$\mathcal{D}^{\prime}_{A/A_0}=\mathbb{Z}(A/A_0).~\eqno(9)$$ 
Since $A,V\in \mathcal{S}^*(\mathcal{D}^{\prime})$, we obtain $A_1^\#=V\cap A\in \mathcal{S}^*(\mathcal{D}^{\prime})$ by Eq.~(1). Every basic set of $\mathcal{D}^{\prime}_A$ is contained in an $A_0$-coset by Eq.~(9). The group $A_1$ is a $\mathcal{D}^{\prime}_A$-subgroup such that $A_1\cap A_0=\{e\}$. So $\mathcal{D}^{\prime}_{A_1}=\mathbb{Z}A_1$. Together with Eq.~(8) and Lemma~\ref{groupring} applied to $A_0\times A_1$, this implies that 
$$\mathcal{D}^{\prime}_{A_0\times A_1}=\mathbb{Z}(A_0\times A_1).~\eqno(10)$$

It is easy to see that $x^2\in U$. So $\{x^2\}\in \mathcal{S}(\mathcal{D}^{\prime})$ by Eq.~(10). Therefore 
$$\{y\}=\{yx^{-2}\}\{x^2\}\in \mathcal{S}(\mathcal{D}^{\prime})~\eqno(11)$$
by Lemma~\ref{basicset} and Eq.~(6). From Eqs.~(10) and~(11) and Lemma~\ref{groupring} applied to~$U=(A_0\times A_1)\rtimes B$ it follows that $\mathcal{D}^{\prime}_U=\mathbb{Z}U$. Together with Eq.~(9), this yields that every basic set of $\mathcal{D}$ is a $\mathcal{D}^{\prime}$-set. Therefore $\mathcal{D}^{\prime}\geq \mathcal{D}$. Thus, $\mathcal{D}^{\prime}=\mathcal{D}$ and we are done. 
\end{proof}

\begin{rem}\label{autgw}
 The group $\aut(\Sigma)=\aut(\mathcal{D})=\aut(\mathbb{Z}U\wr_{U/A_0} \mathbb{Z}(G/A_0))$ is the canonical generalized wreath product of $U_{\r}$ by $(G/A_0)_{\r}$ (see~\cite[Section~5.3]{EP3} for the definitions).
\end{rem}

\begin{lemm}\label{pr33}
The WL-rank of $\Sigma$ is equal to~$6k$.
\end{lemm}

\begin{proof}
From Lemma~\ref{pr32} it follows that $\WL(\Sigma)=\mathcal{D}=\mathbb{Z}U \wr_{U/A_0} \mathbb{Z}(G/A_0)$. Observe that $|U|=|G/A_0|=4k$ and $|U/A_0|=2k$. So $\rk(\mathbb{Z}U)=\rk(\mathbb{Z}(G/A_0))=4k$ and $\rk(\mathbb{Z}(U/A_0))=2k$. Therefore 
$$\rk(\mathcal{D})=\rk(\mathbb{Z}U)+\rk(\mathbb{Z}(G/A_0))-\rk(\mathbb{Z}(U/A_0))=6k$$
by Eq.~(2). Thus, $\rkwl(\Sigma)=\rk(\mathcal{D})=6k$.
\end{proof}

Theorem~\ref{main3} immediately follows from Lemma~\ref{pr31} and Lemma~\ref{pr33}. 

\begin{rem}\label{ddg}
Note that $\Sigma$ is satisfied the condition from~\cite[Theorem~1 (2c)]{Sh}. In particular, $\Sigma$ is divisible design and integral.
\end{rem}

\section{Proof of Theorem~\ref{main4}}

In this section, we use the notations from the previous ones. Let $v(D)\leq 13$. Computer calculations~\cite{HM} imply that the $S$-ring $\mathcal{A}(D)=\WL(\Gamma(D))$ is separable. So the $S$-ring $\mathcal{B}(D,m)=\WL(K_m[\Gamma(D)])\cong\mathcal{A}(D)\wr \mathcal{T}_{C_m}$ is separable for every $m\geq 1$ by Statement~$1$ of Lemma~\ref{trivsep} and Lemma~\ref{separ}. Therefore $\dimwl(K_m[\Gamma(D)])\leq 2$ by Lemma~\ref{dim2}.  

The $S$-ring $\mathcal{T}_{C_4}\otimes \mathcal{T}_{C_m}=\WL(K_4\times K_m)$ is separable for every $m\geq 1$ by Statement~$1$ of Lemma~\ref{trivsep} and Lemma~\ref{separ}. So $\dimwl(K_4\times K_m)\leq 2$ by Lemma~\ref{dim2}.

The $S$-ring $\mathcal{C}_0$ is separable by Statement~$2$ of Lemma~\ref{trivsep}. So $\mathcal{C}(m)=\WL(K_m[K_4\times K_2])\cong \mathcal{C}_0\wr \mathcal{T}_{C_m}$ is separable for every $m\geq 1$ by Lemma~\ref{separ}. Together with Lemma~\ref{dim2}, this yields that $\dimwl(K_m[K_4\times K_2])\leq 2$.

The $S$-ring $\mathcal{D}(k)\cong \mathbb{Z}D_{4k} \wr_{D_{4k}/C_2} \mathbb{Z}D_{4k}$ is separable for every odd $k\geq 3$ by Statement~$3$ of Lemma~\ref{trivsep}. So $\dimwl(\Sigma(k))\leq 2$ by Lemma~\ref{dim2}.

Each of the graphs $K_m[\Gamma(D)]$, $K_4\times K_m$ with $m\geq 2$, $K_m[K_4\times K_2]$, $\Sigma(k)$ has WL-rank at least~$4$. So none of the above graphs is strongly regular. Therefore each of the above graphs has WL-dimension at least~$2$ by Lemma~\ref{dim1}. Together with the previous paragraphs, this implies that each of the considered graphs has WL-dimension~$2$. 

\begin{rem}\label{wldim}
Note that if $v(D)>13$ then $\dimwl(K_m(\Gamma(D))$ can be greater than~$2$. There exist cyclic difference sets with $v(D)=15$ and $v(D)=19$ by~\cite{Bau}. From~\cite{HM} (see also~\cite{vanDam}) it follows that in these cases there exists an association scheme which is algebraically isomorphic but not isomorphic to the Cayley scheme corresponding to $\mathcal{A}(D)$. So $\mathcal{A}(D)$ is not separable and hence $\dimwl(\Gamma(D))>2$ by Lemma~\ref{dim2}. The $S$-ring $\mathcal{A}(D)\wr \mathcal{T}_{C_m}$ is not separable by Lemma~\ref{separ} and hence $\dimwl(K_m[\Gamma(D)])>2$ by Lemma~\ref{dim2}. It would be interesting to find $\dimwl(K_m[\Gamma(D)])$ for arbitrary~$D$ and~$m$.
\end{rem}

\appendix
\setcounter{secnumdepth}{0}
\section{Appendix}

We collect an information on the strictly Deza dihedrants which occur in the paper in the table below. The information on WL-rank is taken from Theorems~\ref{main1}, \ref{main2}, and~\ref{main3}. The parameters of graphs and their WL-closures can be found in Sections~$3$, $4$, and~$5$. The information on WL-dimension follows from Theorem~\ref{main4}. The automorphism groups of graphs were found using equality $\aut(\Gamma)=\aut(\WL(\Gamma))$, Lemma~\ref{aut}, and Remark~\ref{autgw}. The automorphism group of $\mathcal{A}(D)$, where $v(D)\in\{7,11\}$, was computed by~\cite{GAP}.

\begin{table}[h]
{\tiny
\scalebox{0.87}{\begin{tabular}{|l|l|l|l|l|l|}
  \hline
\makecell{graph} & \makecell{parameters} & \makecell{WL-rank} & \makecell{WL-closure} & \makecell{WL-dim} & \makecell{$\aut$} \\
  \hline
  
\makecell{$\Gamma(D)$} & \makecell{$(2n,n-1+k,2k,2(k-1))$,\\ $k=\frac{2n-1-\sqrt{8n-7}}{2}$} & \makecell{$4$} & \makecell{$\mathcal{A}(D)$} & \makecell{$2$ if $n\leq 11$} & \makecell{ $\PSL(3,2)\rtimes C_2$ if $n=7$\\ $\PSL(2,11)\rtimes C_2$ if $n=11$}\\ \hline

\makecell{$K_4\times K_m$} & \makecell{$(4m,m+2,m-2,2)$,\\ $m\geq 2$, $m$ is not divisible by~$4$} & \makecell{$4$} & \makecell{$\mathcal{T}_{C_4}\otimes \mathcal{T}_{C_m}$} & \makecell{$2$} & \makecell{$\sym(4)\times \sym(m)$} \\ \hline

\makecell{$K_m[\Gamma(D)]$} & \makecell{$(2lm,2lm-l+k-1,lm-l+2k,lm-l+2k-2)$,\\ $m\geq 2$, $k=\frac{2l-1-\sqrt{8l-7}}{2}$} & \makecell{$5$} & \makecell{$\mathcal{A}(D)\wr\mathcal{T}_{C_m}$} & \makecell{$2$ if $l\leq 11$} & \makecell{$(\PSL(3,2)\rtimes C_2)\wr\sym(m)$ if $l=7$\\ $(\PSL(2,11)\rtimes C_2)\wr \sym(m)$ if $l=11$} \\ \hline

\makecell{$K_m[K_4\times K_2]$} & \makecell{$(16m,16m-4,16m-6,16m-8)$, $m\geq 2$} & \makecell{$5$} & \makecell{$(\mathcal{T}_{C_2}\otimes \mathcal{T}_{C_4})\wr \mathcal{T}_{C_m}$} & \makecell{$2$} & \makecell{$(\sym(4)\times \sym(2))\wr \sym(m)$}\\ \hline

\makecell{$\Sigma(k)$} & \makecell{$(8k,2(k+1),2(k-1),2)$}  & \makecell{$6k$} & \makecell{$\mathbb{Z}D_{4k} \wr_{D_{4k}/C_2} \mathbb{Z}D_{4k}$} & \makecell{$2$} & \makecell{$D_{4k} \wr_{D_{4k}/C_2} D_{4k}$} \\ \hline

\end{tabular}
}}
\vspace{\baselineskip}
\caption{Strictly Deza dihedrants.}
\end{table}

\end{document}